\documentclass[11pt]{amsart}
\usepackage{amsmath, amssymb, amscd, mathrsfs, url, pinlabel,verbatim}
\usepackage[pagebackref]{hyperref}
\usepackage[margin=1.25in,marginparwidth=0.75in,centering,letterpaper,dvips]{geometry}
\usepackage{color,dcpic,latexsym,graphicx,epstopdf,comment}
\usepackage[all]{xy}
\usepackage[dvipsnames]{xcolor}
\usepackage{tikz,tikz-cd,pgfplots}
\usepackage{upquote,tabularx,textcomp}
\usepackage[shortlabels]{enumitem}
\usepackage{mathtools}
\usepackage[color=blue!20!white,textsize=tiny]{todonotes}

\title{Small Heegaard genus and $\SU(2)$}

\author[John A. Baldwin]{John A. Baldwin}
\address{Department of Mathematics \\ Boston College}
\email{john.baldwin@bc.edu}

\author[Steven Sivek]{Steven Sivek}
\address{Department of Mathematics\\Imperial College London}
\email{s.sivek@imperial.ac.uk}

\makeatletter
\newtheorem*{rep@theorem}{\rep@title}
\newcommand{\newreptheorem}[2]{%
\newenvironment{rep#1}[1]{%
 \def\rep@title{#2 \ref{##1}}%
 \begin{rep@theorem}}%
 {\end{rep@theorem}}}
\makeatother

\newtheorem {theorem}{Theorem}
\newreptheorem{theorem}{Theorem}
\newtheorem {lemma}[theorem]{Lemma}
\newtheorem {proposition}[theorem]{Proposition}
\newtheorem {corollary}[theorem]{Corollary}

\numberwithin{equation}{section}
\numberwithin{theorem}{section}

\theoremstyle{definition}
\newtheorem{definition}[theorem]{Definition}

\newtheorem{remark}[theorem]{Remark}
\newtheorem*{remark*}{Remark}

\newlist{pcases}{enumerate}{1}
\setlist[pcases]{
  label=\bf{Case~\arabic*:}\protect\thiscase.~,
  ref=\arabic*,
  align=left,
  labelsep=0pt,
  leftmargin=0pt,
  labelwidth=0pt,
  parsep=0pt
}
\newcommand{\case}[1][]{%
  \if\relax\detokenize{#1}\relax
    \def\thiscase{}%
  \else
    \def\thiscase{~#1}%
  \fi
  \item
}

\newcommand{\Z}{\mathbb{Z}}

\newcommand{\R}{\mathbb{R}}
\newcommand{\C}{\mathbb{C}}

\newcommand{\F}{\mathbb{F}}
\newcommand{\Q}{\mathbb{Q}}

\newcommand{\Hom}{\operatorname{Hom}}

\newcommand{\Img}{\operatorname{Im}}

\newcommand{\sR}{\mathscr{R}}

\newcommand\SU{\mathrm{SU}}
\newcommand\SL{\mathrm{SL}}
\newcommand\SO{\mathrm{SO}}

\newcommand\KHI{\mathit{KHI}}

\newcommand\HFK{\widehat{\mathit{HFK}}}

\DeclareFontFamily{U}{mathx}{\hyphenchar\font45}
\DeclareFontShape{U}{mathx}{m}{n}{
      <5> <6> <7> <8> <9> <10>
      <10.95> <12> <14.4> <17.28> <20.74> <24.88>
      mathx10
      }{}
\DeclareSymbolFont{mathx}{U}{mathx}{m}{n}
\DeclareFontSubstitution{U}{mathx}{m}{n}
\DeclareMathAccent{\widecheck}{0}{mathx}{"71}

\newcommand{\ad}{\operatorname{ad}}
\newcommand{\Ad}{\operatorname{Ad}}

\newcommand{\tr}{\operatorname{tr}}

\newcommand{\dcover}{\Sigma_2}

\newcommand{\su}{\mathfrak{su}}
\newcommand{\lk}{\mathrm{lk}}

\makeatletter
\DeclareFontFamily{OMX}{MnSymbolE}{}
\DeclareSymbolFont{MnLargeSymbols}{OMX}{MnSymbolE}{m}{n}
\SetSymbolFont{MnLargeSymbols}{bold}{OMX}{MnSymbolE}{b}{n}
\DeclareFontShape{OMX}{MnSymbolE}{m}{n}{
    <-6>  MnSymbolE5
   <6-7>  MnSymbolE6
   <7-8>  MnSymbolE7
   <8-9>  MnSymbolE8
   <9-10> MnSymbolE9
  <10-12> MnSymbolE10
  <12->   MnSymbolE12
}{}
\DeclareFontShape{OMX}{MnSymbolE}{b}{n}{
    <-6>  MnSymbolE-Bold5
   <6-7>  MnSymbolE-Bold6
   <7-8>  MnSymbolE-Bold7
   <8-9>  MnSymbolE-Bold8
   <9-10> MnSymbolE-Bold9
  <10-12> MnSymbolE-Bold10
  <12->   MnSymbolE-Bold12
}{}

\let\llangle\@undefined
\let\rrangle\@undefined
\DeclareMathDelimiter{\llangle}{\mathopen}%
                     {MnLargeSymbols}{'164}{MnLargeSymbols}{'164}
\DeclareMathDelimiter{\rrangle}{\mathclose}%
                     {MnLargeSymbols}{'171}{MnLargeSymbols}{'171}
\makeatother

\newcounter{desccount}

\newcommand{\descref}[1]{\hyperref[#1]{#1}}

\usetikzlibrary{calc,intersections}
\tikzset{every picture/.style=thick}
\tikzset{link/.style = { white, double = black, line width = 1.75pt, double distance = 1.25pt, looseness=1.75 }}
\tikzset{crossing/.style = {draw, circle, dotted, minimum size=0.5cm, inner sep=0, outer sep=0}}
\pgfplotsset{compat=1.12}

\begin{document}

\begin{abstract}
Let $Y$ be a closed, orientable 3-manifold with Heegaard genus 2.  We prove that if $H_1(Y;\Z)$ has order $1$, $3$, or $5$, then there is a representation $\pi_1(Y) \to \SU(2)$ with non-abelian image.  Similarly, if $H_1(Y;\Z)$ has order $2$ then we find a non-abelian representation $\pi_1(Y) \to \SO(3)$.  We also prove that a knot $K$ in $S^3$ is a trefoil if and only if there is a unique conjugacy class of irreducible representations $\pi_1(S^3\setminus K) \to \SU(2)$ sending a fixed meridian to $\left(\begin{smallmatrix}i&0\\0&-i\end{smallmatrix}\right)$.
\end{abstract}

\maketitle

\section{Introduction}

The instanton Floer homology of a homology 3-sphere $Y$ is generated as a chain complex by conjugacy classes of irreducible representations $\pi_1(Y) \to \SU(2)$, so it is natural to ask whether these exist.  This question has been studied from many different perspectives in recent years, including surgeries on knots \cite{km-su2}, branched double covers \cite{zentner-simple}, Stein fillings \cite{bs-stein}, and incompressible tori \cite{lpcz,bs-splicing}; in every case so far the answer has been yes, unless $Y\cong S^3$.

In this article, we add small Heegaard genus to the list.  If $Y$ has a Heegaard splitting of genus at most 1 then $\pi_1(Y)$ is abelian, so the first interesting case is Heegaard genus 2.  In this case we can apply work of Birman and Hilden \cite{birman-hilden-branched} to identify such manifolds as branched covers of links, and thus reduce the problem to one about representations of link groups.  Our main result is the following.

\begin{theorem} \label{thm:heegaard-genus-2}
Let $Y$ be a rational homology 3-sphere with Heegaard genus 2, and with first homology $H_1(Y;\Z)$ of order $1$, $3$, or $5$.  Then there is a representation $\pi_1(Y) \to \SU(2)$ with non-abelian image.
\end{theorem}

In general, we will say that $Y$ is \emph{$\SU(2)$-abelian} if every homomorphism $\pi_1(Y) \to \SU(2)$ has abelian image; then manifolds satisfying the hypotheses of Theorem~\ref{thm:heegaard-genus-2} are \emph{not} $\SU(2)$-abelian.

One could ask the same question about rational homology 3-spheres with other homology groups: for fixed $n \geq 1$, if $Y$ has Heegaard genus 2 and $|H_1(Y)|=n$, must there be a non-abelian representation $\pi_1(Y) \to \SU(2)$?  The answer is negative for many other values of $n$.  For example:
\begin{itemize}
\setlength{\itemsep}{0.5em}
\item If $n \geq 4$ is even, then the connected sum $\mathbb{RP}^3 \# L(n/2,1)$ is $\SU(2)$-abelian.
\item For $n=|24g+11|$, where $g\neq 0,-1$, the toroidal manifolds $Y=Y(T_{2,3},T_{2,2g+1})$ studied in \cite{zentner-simple} (and originally constructed by Motegi \cite[\S3]{motegi}) are $\SU(2)$-abelian.  These have Heegaard genus 2 because they are branched double covers of 3-bridge knots $L(T_{2,3},T_{2,2g+1})$ \cite[Theorem~4.14]{zentner-simple}; see \cite[Figure~1]{sz-menagerie}. For example, when $n=35$ the branch locus is identified in \cite[\S5]{zentner-simple} as the knot $8_{16}$.
\item The smallest known hyperbolic example, pointed out to us by Nathan Dunfield, is $\mathrm{Vol3} = \mathrm{m007}(3,1)$, with $n=18$.  It has Heegaard genus 2 because it is the branched double cover of the $3$-bridge link $L10n46$.
\item The smallest known hyperbolic example with $n$ odd is $\mathrm{m036}(-3,2)$, with $n=45$; it is the branched double cover of the 3-bridge knot $8_{18}$, or equivalently the 4-fold branched cyclic cover of the figure eight knot. (See \cite[\S6]{cornwell}, or \cite[Example~5.6]{sz-menagerie} for more discussion.)  It is also the unique double cover of $\mathrm{Vol3}$.
\end{itemize}

Unfortunately, we do not expect that the strategy we use to prove Theorem~\ref{thm:heegaard-genus-2} will work for larger odd values of $|H_1(Y)|$, because the order-$5$ case already requires several special results that do not apply to higher orders.  One is a theorem of Li and Liang \cite[Theorem~1.4]{li-liang} asserting that if the instanton knot homology $\KHI(K)$ of \cite{km-excision} has rank 1 in Alexander gradings $\pm g$, $\pm(g-1)$, and $0$, and it vanishes everywhere else, then either $K$ or its mirror is an \emph{instanton L-space knot} of genus $g$.  The analogous claim for $\HFK(K)$ is certainly true but does not seem to exist in the Heegaard Floer literature.  The second theorem, by Farber, Reinoso, and Wang \cite{frw-cinquefoil}, says that the $(2,5)$ torus knot is the only instanton L-space knot of genus 2.  We would like to generalize the first to larger values of $\dim \KHI(K)$, and the second to genera $g \geq 3$, but both of these seem entirely out of reach.

Theorem~\ref{thm:heegaard-genus-2} makes use of some facts about $\SU(2)$-simple knots, whose definition we recall now.  Here and in the sequel we view $\SU(2)$ as the unit quaternions; this identifies the trace of a $2\times 2$ matrix with twice the real part of the corresponding quaternion.

\begin{definition}[\cite{zentner-simple}] \label{def:su2-simple}
Given a link $L \subset S^3$, a representation
\[ \rho: \pi_1(S^3 \setminus L) \to \SU(2) \]
is \emph{meridian-traceless} if $\operatorname{Re}(\rho(\mu)) = 0$ for every meridian $\mu$ of $L$.  We then say that $L$ is \emph{$\SU(2)$-simple} if every meridian-traceless representation is conjugate to one with image in the binary dihedral group $\{e^{i\theta}\} \cup \{ e^{i\theta}j \}$.
\end{definition}

What we really show en route to proving Theorem~\ref{thm:heegaard-genus-2} is that the only $\SU(2)$-simple knots of determinant 1 or 3 are the unknot and the trefoils $T_{\pm2,3}$, and that the only $\SU(2)$-simple knots of determinant $5$ and bridge index at most $3$ are the figure eight and the cinquefoils $T_{\pm2,5}$.

Building on this, we can study the representation variety
\[ \sR(K,i) = \left\{ \rho: \pi_1(S^3 \setminus K) \to \SU(2) \mid \rho(\mu) = i \right\}, \]
which consists of representations that send a fixed meridian $\mu$ of $K$ to the quaternion $i$.  Kronheimer and Mrowka \cite[Theorem~7.17]{km-excision} proved that $\sR(K,i) \cong \{\ast\}$ if and only if $K$ is the unknot.  For the trefoils we have
\[ \sR(T_{\pm2,3},i) \cong \{ \ast \} \sqcup S^1, \]
and \cite[Conjecture~1.9]{bs-trefoil} says that this should uniquely characterize the trefoils. The classification of $\SU(2)$-simple knots of determinant $3$ allows us to prove that this is indeed the case.

\begin{theorem} \label{thm:trefoil-recognition-main}
Suppose for a knot $K \subset S^3$ that $\sR(K,i) \cong \{ \ast \} \sqcup S^1$.  Then $K$ is a trefoil.
\end{theorem}

Partial results toward Theorem~\ref{thm:trefoil-recognition-main} were known under the assumption that the $S^1$ component consists of \emph{nondegenerate} representations: Kronheimer and Mrowka \cite[Corollary~7.20]{km-excision} first proved that $K$ must be fibered, and then in \cite[Theorem~1.10]{bs-trefoil} we proved that $K$ must be a trefoil.  Theorem~\ref{thm:trefoil-recognition-main} removes the nondegeneracy assumption, establishing the conjecture in full generality.

Along the way to classifying $\SU(2)$-simple knots of small determinant, we deduce the following result which may be independently useful, as part of Theorem~\ref{thm:simple-alexander}.

\begin{theorem} \label{thm:main-simple-alexander}
Let $K\subset S^3$ be an $\SU(2)$-simple knot whose determinant is prime, and write its Alexander polynomial as
\[ \Delta_K(t) = \sum_{i=-g}^g a_i t^i \]
where $g$ is the Seifert genus of $K$.  Then $(-1)^{i+\sigma(K)/2}a_i \geq 0$ for all $i$, where $\sigma(K)$ is the signature of $K$, and $|a_g| \geq 1$ with equality if and only if $K$ is fibered.
\end{theorem}

In fact, in Theorem~\ref{thm:simple-alexander} we show that $(-1)^{i+\sigma(K)/2}a_i$ is nonnegative because it is equal to the dimension of the summand $\KHI(K,i)$ of the instanton knot homology of $K$.

The reason we insist that $\det(K)$ be prime in Theorem~\ref{thm:main-simple-alexander} is that it allows us to verify a nondegeneracy condition for binary dihedral representations of $\pi_1(S^3\setminus K)$, following work of Heusener and Klassen \cite{heusener-klassen}; we do not know whether the same conclusions should hold when $K$ is $\SU(2)$-simple but $\det(K)$ is composite.  These nondegeneracy conditions allow us to show that $\dim \KHI(K) = \det(K)$, and then in the determinant-3 case, we know from \cite{bs-trefoil} that the only such knots are the trefoils.  When $\det(K)=5$ this is not quite enough, and we also need to understand Stein fillings of $\SU(2)$-abelian manifolds; see Proposition~\ref{prop:fillings} for details.

Finally, the reader might wonder what happens when $|H_1(Y)|$ is even.  In this case we have to consider $\SU(2)$-simplicity for links rather than knots, and less is known in this setting, but recent work of Xie and Zhang \cite{xie-zhang-traceless} provides a starting point.  Even so, it is hard to guarantee the existence of $\SU(2)$ representations, because our methods might lead to $\SO(3)$ representations that do not lift to $\SU(2)$.  (When $|H_1(Y)|$ is odd, there is no obstruction to lifting.)   But we can often at least find non-abelian $\SO(3)$ representations, and in particular we prove the following.

\begin{theorem} \label{thm:order-2}
If $Y$ has Heegaard genus $2$ and satisfies $|H_1(Y;\Z)| = 2$, then there is a representation
\[ \pi_1(Y) \to \SO(3) \]
with non-abelian image.
\end{theorem}

As in the case of determinants $1$ and $3$, what we really prove is Proposition~\ref{prop:su2-simple-det-2}, asserting that the Hopf link is the only $\SU(2)$-simple link of determinant $2$, and then Theorem~\ref{thm:order-2} follows from the special case of 3-bridge links.

\subsection*{Organization}

In Section~\ref{sec:zhs3} we prove Theorem~\ref{thm:heegaard-genus-2} for integral homology spheres.  In Section~\ref{sec:khi-simple} we study the instanton knot homology of $\SU(2)$-simple knots, and in Section~\ref{sec:fillings} we prove Proposition~\ref{prop:fillings}, asserting that simply connected Stein fillings of many $\SU(2)$-abelian manifolds must have negative definite intersection form.  In Section~\ref{sec:35} we apply these to prove Theorem~\ref{thm:heegaard-genus-2} in the cases where $H_1(Y)$ has order 3 or 5.  Section~\ref{sec:trefoil} uses this work to prove Theorem~\ref{thm:trefoil-recognition-main}, that the variety $\sR(K,i)$ detects the trefoils.  Finally, Section~\ref{sec:order-2} is devoted to the proof of Theorem~\ref{thm:order-2}.

\subsection*{Acknowledgments}

We are grateful to Tye Lidman, Mike Miller Eismeier, Danny Ruberman, and Raphael Zentner for interesting conversations regarding this work.  We also thank Nathan Dunfield for telling us about $\mathrm{Vol3}$ and its character variety, and the anonymous referee for their feedback.  JAB was supported by NSF FRG Grant DMS-1952707.

\section{From $\SU(2)$-simple knots to $\SU(2)$-abelian 3-manifolds} \label{sec:zhs3}

Given a knot $K \subset S^3$, we will let $\dcover(K)$ denote the branched double cover of $K$.  The proof of Theorem~\ref{thm:heegaard-genus-2} then relies on the following proposition.  

\begin{proposition} \label{prop:dcover-map}
For each odd integer $d\geq 1$, the map $K \mapsto \dcover(K)$ gives a well-defined surjection
\begin{equation} \label{eq:main-surjection}
\left\{ \begin{array}{c} \SU(2)\text{-simple 3-bridge knots} \\ K\subset S^3\text{ of determinant }d \end{array} \right\}
\to
\left\{ \begin{array}{c} \SU(2)\text{-abelian} \text{ 3-manifolds $Y$ with} \\ \text{Heegaard genus 2 and }|H_1(Y)| = d \end{array} \right\}.
\end{equation}
\end{proposition}

\begin{proof}
We first check that this map is well-defined.  If $K$ is an $\SU(2)$-simple 3-bridge knot, then its branched double cover is $\SU(2)$-abelian by \cite[Theorem~1.6]{sz-menagerie}.  Moreover, $\dcover(K)$ certainly has Heegaard genus at most 2, since we can split $K$ in half using a bridge sphere, and the branched double cover of the 3-ball over either half will be a genus-2 handlebody.  In fact, the Heegaard genus of $\dcover(K)$ must be exactly 2, since otherwise $\dcover(K)$ would be $S^3$ or a lens space, and then Hodgson and Rubinstein \cite{hodgson-rubinstein} proved that $K$ would have bridge index at most 2.  So this proves that $\dcover(K)$ belongs to the codomain, as claimed.

We now wish to show that the map is surjective, so let $Y$ be an element of the codomain.  Since $Y$ has Heegaard genus 2, Birman and Hilden \cite[Theorem~1]{birman-hilden-branched} proved that it is the branched double cover of a 3-bridge link $L \subset S^3$.  Writing $\F = \Z/2\Z$, we have a short exact sequence of chain complexes
\[ 0 \to C_*(S^3,L;\F) \to C_*(\dcover(L);\F) \to C_*(S^3;\F) \to 0, \]
where the first map is a transfer map and the second is induced by projection, and this leads to an isomorphism
\begin{equation} \label{eq:homology-bdc}
H_1(\dcover(L);\F) \cong H_1(S^3,L;\F) \cong \tilde{H}_0(L;\F)
\end{equation}
using the long exact sequence of the pair $(S^3,L)$.   Thus if $L$ has $\ell$ components then
\[ 0 = \dim_\F H_1(Y;\F) = \dim_\F H_1(\dcover(L); \F) = \ell - 1, \]
so $\ell=1$ and we can write $Y = \dcover(K)$ where $K=L$ is a 3-bridge knot.

We know from \cite[Lemma~3.2]{zentner-simple} and the fact that $H_1(\dcover(K);\F)=0$ that $\dcover(K)$ is $\SU(2)$-abelian if and only if it is $\SO(3)$-abelian, and the latter implies that $K$ is $\SU(2)$-simple by \cite[Proposition~3.1]{zentner-simple}, so if $Y$ is $\SU(2)$-abelian then $K$ must be $\SU(2)$-simple.  Thus $K$ belongs to the preimage of $Y$, proving the desired surjectivity.
\end{proof}

\begin{remark} \label{rem:abelian-implies-simple}
We note from the proof of Proposition~\ref{prop:dcover-map} (and in particular from \cite[\S3]{zentner-simple}) the fact that if a knot $K\subset S^3$ has an $\SU(2)$-abelian branched double cover, then $K$ must be $\SU(2)$-simple.  This holds regardless of the bridge index of $K$ or the Heegaard genus of $\dcover(K)$.  By contrast, an $\SU(2)$-simple knot of bridge index greater than 3 need not have an $\SU(2)$-abelian branched double cover.
\end{remark}

The following fact constrains the knots considered in Proposition~\ref{prop:dcover-map}, though we will not use it in this paper.

\begin{proposition} \label{prop:hyperbolic}
All $\SU(2)$-simple 3-bridge knots are hyperbolic.
\end{proposition}

\begin{proof}
Suppose that $K$ is an $\SU(2)$-simple 3-bridge knot, but that it is not hyperbolic.  Then $K$ is either a torus knot or a satellite \cite[Corollary~2.5]{thurston-kleinian}, so we address these cases separately below.

Suppose first that $K$ is a torus knot.  Schubert \cite[Satz~10]{schubert} showed that the only 3-bridge torus knots are $T(3,q)$ where $|q| > 3$, so $K = T(3,q)$ for some such $q$.  Then $\det(K)$ is either $1$ or $3$ depending on whether $q$ is odd or even, so $\dcover(K)$ is an $\SU(2)$-abelian Seifert fibered space, with first homology of order 1 or 3.  The only such manifolds are $S^3$ and lens spaces \cite[Theorem~1.2]{sz-menagerie}, contradicting the fact that $\dcover(K)$ has Heegaard genus 2.

Now suppose that $K$ is a satellite knot, say with pattern $P \subset S^1\times D^2$ and companion $C$.  Let $\alpha > 0$ be the wrapping number of $P$, meaning the minimal number of times that some meridional disk of $S^1\times D^2$ intersects $P$.  Then Schubert \cite[Satz~3]{schubert} proved that the bridge index satisfies
\[ b(K) \geq \alpha \cdot b(C). \]
Since $b(K)=3$ and the knottedness of $C$ implies that $b(C) \geq 2$, we must have $\alpha=1$, hence $K$ is a connected sum $C \# C'$.  But then $b(C\#C')=b(C) + b(C') - 1$ by \cite[Satz~7]{schubert}, so $C$ and $C'$ are both 2-bridge knots, and thus the $\SU(2)$-abelian manifold
\[ \dcover(K) \cong \dcover(C) \# \dcover(C') \]
is a connected sum of two lens spaces of odd order.  Now $\pi_1(\dcover(K)) \cong (\Z/p\Z) \ast (\Z/q\Z)$ for some odd integers $p,q \geq 3$, and there is an irreducible representation $\pi_1(\dcover(K)) \to \SU(2)$ defined by sending the generators of the $\Z/p\Z$ and $\Z/q\Z$ factors to $e^{2\pi i/p}$ and $e^{2\pi j/q}$ respectively, so $\dcover(K)$ is not $\SU(2)$-abelian after all and we have a contradiction.
\end{proof}

Proposition~\ref{prop:dcover-map}, together with some gauge-theoretic input from Kronheimer and Mrowka \cite{km-excision}, suffices to establish the homology sphere case of Theorem~\ref{thm:heegaard-genus-2}.

\begin{proof}[Proof of Theorem~\ref{thm:heegaard-genus-2} in the case $H_1(Y;\Z)=0$]
We wish to show that there are no $\SU(2)$-abelian integer homology 3-spheres of Heegaard genus 2.  By Proposition~\ref{prop:dcover-map}, it suffices to show that there are no nontrivial $\SU(2)$-simple knots of determinant 1, so let us suppose that such a knot $K$ exists.  Then Klassen \cite[Theorem~10]{klassen} proved that there are exactly $\frac{1}{2}(\det(K)-1) = 0$ conjugacy classes of non-abelian representations $\pi_1(S^3\setminus K) \to \SU(2)$ with binary dihedral image.  By assumption, every representation
\[ \rho: \pi_1(S^3\setminus K) \to \SU(2) \]
such that $\operatorname{Re}\rho(\mu)=0$ is binary dihedral, so now every such $\rho$ has abelian image.  But Kronheimer and Mrowka \cite[Corollary~7.17]{km-excision} proved that since $K$ is a nontrivial knot, there is at least one non-abelian $\rho$ with $\operatorname{Re}\rho(\mu)=0$, a contradiction.
\end{proof}

\begin{remark}
Zentner \cite[Proposition~9.1]{zentner-simple} gave a slightly different proof of the fact that a homology sphere which is the branched double cover of a knot cannot be $\SU(2)$-abelian.  We argue differently here because the idea of counting binary dihedral representations will turn out to be useful for manifolds with nontrivial first homology.
\end{remark}

\section{Instanton knot homology and $\SU(2)$-simple knots} \label{sec:khi-simple}

When $Y$ is not an integer homology sphere, we can still occasionally say something about representations $\pi_1(Y) \to \SU(2)$, though it will require substantially more machinery.  In preparation, we study the instanton knot homology \cite{km-excision} of $\SU(2)$-simple knots with prime determinant.

\begin{proposition} \label{prop:odd-prime-determinant}
Let $K\subset S^3$ be an $\SU(2)$-simple knot, with determinant an odd prime $p$.  Then its signature satisfies $|\sigma(K)| \leq p-1$ and $\sigma(K) \equiv p-1 \pmod{4}$, and its instanton knot homology satisfies $\dim_\Q \KHI(K) = p$.
\end{proposition}

\begin{proof}
Since $K$ is $\SU(2)$-simple, all of the irreducible $\SU(2)$ representations $\rho: \pi_1(S^3\setminus K) \to \SU(2)$ with $\operatorname{Re}(\rho(\mu))=0$ are conjugate to binary dihedral ones, and again by \cite[Theorem~10]{klassen} there are $n=\frac{1}{2}(\det(K)-1)$ of these up to conjugacy.  Thus by \cite[Proposition~7.3]{zentner-simple}, we will have
\[ \dim_\Q I^\natural(K) = 2n+1 = \det(K), \]
where $I^\natural(K)$ is the reduced singular instanton knot homology of \cite{km-unknot}, if each such representation $\rho$ satisfies the following two conditions:
\begin{enumerate}[label=(\alph*)]
\item $\dim_\R H^1(S^3 \setminus K; \su(2)_\rho) = 1$, and 
\item the restriction map $H^1(S^3\setminus K; \su(2)_\rho) \to H^1(\mu; \su(2)_\rho)$ is onto. 
\end{enumerate}

If $\rho$ is one of these irreducible, binary dihedral representations, with image in $\{e^{i\theta}\} \cup \{e^{i\theta}j\}$, then one can extract from the proof of \cite[Theorem~10]{klassen} that it has finite image, of order dividing $4\det(K)=4p$.  (For completeness, we will work out the details explicitly in Lemma~\ref{lem:dihedral-image-link} and Remark~\ref{rem:dihedral-image-knot}.)  Since $\rho(\mu)$ is traceless it has order $4$, so then the order of $\Img(\rho)$ is also a multiple of 4, and it is not equal to $4$ or else $\Img(\rho)$ would be abelian.  Then $\Img(\rho)$ has order exactly $4p$ since $p$ is prime.  Now Proposition~\ref{prop:binary-dihedral-nondegenerate} below asserts that $\rho$ satisfies both of the desired conditions.

We deduce from this that $\dim_\Q I^\natural(K) = \det(K) = p$ after all, and then the isomorphism $I^\natural(K;\Q) \cong \KHI(K;\Q)$ of \cite[Proposition~1.4]{km-unknot} tells us that $\dim_\Q \KHI(K) = p$ as claimed.  These computations of the twisted cohomology of $S^3\setminus K$ also suffice to ensure that the Casson--Lin invariant $h(K)$, defined in \cite{lin-casson} and reinterpreted using gauge theory by Herald \cite{herald}, is in fact a signed count of the conjugacy classes of irreducible representations with $\operatorname{Re}(\rho(\mu)) = 0$.  It follows that
\[ |h(K)| \leq \tfrac{1}{2}\left(\det(K)-1\right), \]
but Lin \cite[Corollary~2.10]{lin-casson} proved that $h(K)=\frac{1}{2}\sigma(K)$, so this establishes the inequality $|\sigma(K)| \leq \det(K)-1$.  Finally, the congruence $\sigma(K) \equiv \det(K)-1 \pmod{4}$ actually holds for arbitrary knots $K \subset S^3$, as an immediate consequence of the evenness of $\sigma(K)$ and the relation
\begin{equation} \label{eq:det-sign}
\det(K) \equiv (-1)^{\sigma(K)/2} \pmod{4}
\end{equation}
of \cite[Theorem~5.6]{murasugi-signature}, so the proof is complete.
\end{proof}

The following proposition completes the missing part of Proposition~\ref{prop:odd-prime-determinant}.  The hard technical details of its proof are essentially contained in \cite{heusener-klassen}, though we have to do some work to extract them.

\begin{proposition} \label{prop:binary-dihedral-nondegenerate}
Let $K \subset S^3$ be a knot, and let $\rho: \pi_1(S^3 \setminus K) \to \SU(2)$ be a representation whose image is binary dihedral of order $4p$ for some odd prime $p$.  Then 
\begin{enumerate}[label=(\alph*)]
\item $\dim_\R H^1(S^3 \setminus K; \su(2)_\rho) = 1$, and \label{i:nondeg-1}
\item the restriction map $H^1(S^3\setminus K; \su(2)_\rho) \to H^1(\mu; \su(2)_\rho)$ is onto. \label{i:nondeg-2}
\end{enumerate}
\end{proposition}

\begin{proof}
Since $\rho: \pi_1(S^3\setminus K) \to \SU(2)$ has binary dihedral image of order $4$ times an odd prime, Heusener and Klassen \cite[\S3]{heusener-klassen} proved that $H^1(S^3\setminus K; \su(2)_\rho)$ is 1-dimensional, verifying \ref{i:nondeg-1}.  In fact, their proof also implies that the restriction map
\begin{equation} \label{eq:h1-restriction}
H^1(S^3\setminus K; \su(2)_\rho) \to H^1(\mu; \su(2)_\rho) \cong \R
\end{equation}
is onto, as we will now explain.  To set the stage, we assume without loss of generality that $\rho(\mu) = j$.  Then $H^1(S^3\setminus K; \su(2)_\rho)$ is the group $Z^1(S^3\setminus K; \su(2)_\rho)$ of cocycles, which consists of functions
\[ \xi: \pi_1(S^3\setminus K) \to \su(2) \]
such that
\[ \xi(gh)=\xi(g)+\Ad_{\rho(g)}\xi(h) \text{ for all }g,h, \]
modulo the group $B^1(S^3\setminus K;\su(2)_\rho)$ of coboundaries, which have the form
\begin{align*}
\zeta^\#: \pi_1(S^3\setminus K) &\to \su(2) \\
g &\mapsto \zeta - \Ad_{\rho(g)}\zeta
\end{align*}
for some $\zeta \in \su(2)$.  (Here we view $\su(2)$ as the vector space of purely imaginary quaternions.)  If $\zeta = ai+bj+ck$, then since $\rho(\mu)=j$ we have
\begin{align*}
\zeta^\#(\mu) &= (ai+bj+ck) - \Ad_j(ai+bj+ck) \\
&= (ai+bj+ck) - (-ai+bj-ck) = 2(ai+ck)
\end{align*}
so the $j$-component of $\zeta^\#(\mu)$ is zero.  It follows that if $\xi \in Z^1(S^3\setminus K;\su(2)_\rho)$ is a 1-cocycle for which $\xi(\mu)$ has nonzero $j$-component, then its restriction to an element of $Z^1(\mu;\su(2)_\rho)$ is not a coboundary, so this is how we will prove the surjectivity of \eqref{eq:h1-restriction}.

In order to determine $Z^1(S^3\setminus K; \su(2)_\rho)$, Heusener and Klassen take a Wirtinger presentation \cite[Equation~(5)]{heusener-klassen} of the form
\[ \pi_1(S^3\setminus K) = \langle S_1,\dots,S_n \mid S_{j_\ell}^{\epsilon_\ell}S_{\ell}S_{j_\ell}^{-\epsilon_\ell} = S_{\ell+1}, \ \ell=1,\dots,n-1 \rangle \]
with each $\epsilon_\ell\in\{\pm 1\}$, and turn this into a system of $n-1$ linear equations in $n+1$ real variables $s_1,\dots,s_n,t$, given in \cite[Equation~(14)]{heusener-klassen} as
\begin{equation} \label{eq:hk-system}
\alpha(j_\ell,\ell)s_{j_\ell} - s_\ell - s_{\ell+1} - \epsilon_{j_\ell}\beta(j_\ell,\ell)t = 0, \qquad 1 \leq \ell \leq n-1.
\end{equation}
Writing $\rho(S_\ell) = \zeta_\ell j$ for a complex unit $\zeta_\ell$ as in \cite[Equation~(6)]{heusener-klassen}, they show that a collection of values $\xi(S_\ell) = x_\ell i + y_\ell j + z_\ell k$ determines a cocycle if and only if
\begin{itemize}
\item we have $x_1 = x_2 = \dots = x_n$ (coming from \cite[Equation~(9)]{heusener-klassen}), and
\item if we write $y_\ell + z_\ell i = \zeta_\ell(t_\ell + s_\ell i)$, then $t_1 = t_2 = \dots = t_\ell$ have some common value $t$, and $(s_1,\dots,s_n,t)$ is a solution to \eqref{eq:hk-system}.
\end{itemize}
Now $\dim H^1(S^3\setminus K;\su(2)_\rho) = 1$ implies that $\dim Z^1(S^3\setminus K;\su(2)_\rho) = 4$, as explained after \cite[Remark~5]{heusener-klassen}, so the space of solutions to \eqref{eq:hk-system} is $3$-dimensional.

Suppose we fix $\mu = S_1$ and $\rho(\mu)=j$ as above, and let $\xi \in Z^1(S^3\setminus K; \su(2)_\rho)$ be a cocycle corresponding to a solution $(s_1,\dots,s_n,t)$ of \eqref{eq:hk-system}.  Following the above notation, we have $\zeta_1 = 1$ and so $y_1+z_1i = t_1+s_1i$; this means that the $j$-coordinate $y_1$ of $\xi(\mu)$ is nonzero if and only if $t_1=t$ is.  Thus it will suffice to find a solution of \eqref{eq:hk-system} with $t \neq 0$, since this will give a cocycle $\xi \in Z^1(S^3\setminus K; \su(2)_\rho)$ whose restriction to $Z^1(\mu;\su(2)_\rho)$ is not a coboundary.

The proof of \cite[Theorem~1]{heusener-klassen} concludes by determining that the simpler system
\begin{equation} \label{eq:hk-system-simplified}
\alpha(j_\ell,\ell)s_{j_\ell} - s_\ell - s_{\ell+1} = 0, \qquad 1 \leq \ell \leq n-1
\end{equation}
of $n-1$ equations in $n$ variables has at most a 2-dimensional solution space (its rank is at least $n-2$).  If every solution to \eqref{eq:hk-system} has $t=0$, then it follows that the space of solutions to \eqref{eq:hk-system} is also at most 2-dimensional, since (by forgetting the $t$ coordinate) it injects into the space of solutions to \eqref{eq:hk-system-simplified}.  But we know that \eqref{eq:hk-system} actually has a 3-dimensional solution space, so there must be a solution with $t\neq 0$.  Thus the map \eqref{eq:h1-restriction} must be onto after all, verifying \ref{i:nondeg-2}.
\end{proof}

We can use Proposition~\ref{prop:odd-prime-determinant} to say a bit more about the instanton knot homology of knots which satisfy its hypotheses.  We first recall that if $K$ has genus $g$, then its instanton knot homology comes with a symmetric Alexander decomposition
\[ \KHI(K) = \bigoplus_{i=-g}^g \KHI(K,i), \]
each of whose summands has a natural $\Z/2\Z$-grading.  Kronheimer--Mrowka \cite{km-alexander} and Lim \cite{lim} proved that this recovers the Alexander polynomial $\Delta_K(t)$, by the formula
\begin{equation} \label{eq:khi-alexander}
\pm \Delta_K(t) = \sum_{i=-g}^g \chi(\KHI(K,i)) \cdot t^i,
\end{equation}
where the sign ambiguity comes from using different conventions for the grading.  

\begin{theorem} \label{thm:simple-alexander}
Let $K \subset S^3$ be an $\SU(2)$-simple knot for which $\det(K)$ is prime.  Then each Alexander summand $\KHI(K,i)$ is supported in a single $\Z/2\Z$ grading.  Moreover, if $K$ has genus $g$ and Alexander polynomial
\[ \Delta_K(t) = \sum_{i=-g}^g a_i t^i, \]
then the following must be true:
\begin{itemize}
\item $a_i = (-1)^{i+\sigma(K)/2} \dim_\Q \KHI(K,i)$ for all $i$, and
\item $|a_g| \geq 1$, with equality if and only if $K$ is fibered.
\end{itemize}
\end{theorem}

\begin{proof}
Proposition~\ref{prop:odd-prime-determinant} says that $\dim_\Q \KHI(K) = \det(K)$.  Now substituting $t=-1$ into \eqref{eq:khi-alexander}, taking absolute values, and applying the triangle inequality gives
\bgroup
\allowdisplaybreaks
\begin{align*}
\det(K) &= \left| \sum_{i=-g}^g \chi(\KHI(K,i))\cdot(-1)^i \right| \\
&\leq \sum_{i=-g}^g \left|(-1)^i\chi(\KHI(K,i))\right| \\
&\leq \sum_{i=-g}^g \dim_\Q \KHI(K,i) \\
&= \det(K).
\end{align*}
\egroup
In particular, equality must hold in both of the inequalities above.  We deduce from the second one that $|\chi(\KHI(K,i))| = \dim_\Q \KHI(K,i)$ for each $i$, so each $\KHI(K,i)$ is supported in a single mod 2 grading.  The first one says that all of the nonzero terms $(-1)^i \chi(\KHI(K,i))$ must have the same sign, say $(-1)^{\delta_0}$, and then the mod 2 grading of each $\KHI(K,i)$ is $i+\delta_0$.  Then
\[ \chi(\KHI(K,i)) = (-1)^{i+\delta_0} \dim_\Q \KHI(K,i), \]
so we substitute this information back into \eqref{eq:khi-alexander} to get
\[ \pm\Delta_K(t) = \sum_{i=-g}^g \left((-1)^{i+\delta_0} \dim_\Q \KHI(K,i)\right) t^i. \]
Let $\delta$ be either $\delta_0$ or $\delta_0+1$, depending on whether the left side is $+\Delta_K(t)$ or $-\Delta_K(t)$.  Then
\[ a_i = (-1)^{i+\delta} \dim_\Q \KHI(K,i) \]
for all $i$, and so we have
\[ \Delta_K(-1) = \sum_{i=-g}^g (-1)^i a_i = (-1)^\delta \sum_{i=-g}^g \dim_\Q \KHI(K,i) = (-1)^\delta \cdot \det(K). \]
But we also know in general that $\Delta_K(-1) = (-1)^{\sigma(K)/2} \det(K)$: this follows from the congruence \eqref{eq:det-sign} and the fact that
\[ \Delta_K(-1) = a_0 + \sum_{i=1}^g\left( (-1)^i + (-1)^{-i} \right)a_i = a_0 + 2 \sum_{i=1}^g (-1)^i a_i \]
is congruent modulo 4 to $a_0 + 2\sum_{i=1}^g a_i = \Delta_K(1) = 1$.  Thus we have $(-1)^\delta = (-1)^{\sigma(K)/2}$, completing the determination of $a_i$ in terms of $\KHI(K,i)$.

The last detail is the claim about $|a_g|$, which is equal to $\dim_\Q \KHI(K,g)$.  We know that this dimension is strictly positive \cite[Proposition~7.6]{km-excision}, and that equality holds if and only if $K$ is fibered \cite[Proposition~4.1]{km-alexander}, so the claim follows immediately.
\end{proof}

%
%

\section{Symplectic fillings of $\SU(2)$-abelian manifolds} \label{sec:fillings}

An \emph{instanton L-space} is a rational homology 3-sphere $Y$ whose framed instanton homology has the smallest possible rank, namely $\dim_\Q I^\#(Y) = |H_1(Y)|$.  By analogy with Heegaard Floer homology \cite[Theorem~1.4]{osz-genus}, we expect that any symplectic filling of an instanton L-space should have negative definite intersection form.  We do not prove this here, but in this section we will prove this claim for a restricted class of instanton L-spaces which will suffice for our purposes.  We begin with the following definition.

\begin{definition}[\cite{boyer-nicas}] \label{def:cyclically-finite}
An $\SU(2)$-abelian rational homology $3$-sphere $Y$ has \emph{cyclically finite} fundamental group if for every representation
\[ \rho: \pi_1(Y) \to \SU(2), \]
the finite cover $\tilde{Y}$ whose fundamental group $\pi_1(\tilde{Y})$ is the kernel of 
\[ \ad\rho: \pi_1(Y) \to \SO(3) \]
is a rational homology sphere.
\end{definition}

For example, in Definition~\ref{def:cyclically-finite} it suffices to know that the universal abelian cover of $Y$ is a rational homology sphere.  The following are Proposition~4.9 and Corollary~4.10 of \cite{bs-stein}.

\begin{lemma}[\cite{bs-stein}] \label{lem:cyclically-finite}
Suppose that $Y$ is a $\SU(2)$-abelian rational homology sphere, and that $|H_1(Y)| \leq 5$ or that $H_1(Y)$ is cyclic of order a prime power.  Then $\pi_1(Y)$ is cyclically finite.
\end{lemma}

In particular, \cite[Theorem~4.6]{bs-stein} says that if $Y$ is an $\SU(2)$-abelian rational homology sphere and $\pi_1(Y)$ is cyclically finite, then $Y$ is an instanton L-space.  We now prove an analogue of \cite[Theorem~1.4]{osz-genus} for such manifolds.

\begin{proposition} \label{prop:fillings}
Let $Y$ be an $\SU(2)$-abelian rational homology sphere, and suppose that $\pi_1(Y)$ is cyclically finite.  Then any simply connected, weak symplectic filling $(X,\omega)$ of a contact structure on $Y$ must satisfy $b^+_2(X) = 0$.
\end{proposition}

\begin{proof}
Let $(X,\omega)$ be a weak symplectic filling of $(Y,\xi)$, with $\pi_1(X) = 1$, and suppose that $b^+_2(X) \geq 1$.  We attach a Weinstein 2-handle $H$ along a Legendrian torus knot $T_{2,5}$ with Thurston--Bennequin invariant $3$, inside a Darboux ball $B^3 \subset Y$, to get a symplectic cobordism to a new $(Y',\xi')$.  Then following Eliashberg \cite{eliashberg-cap}, we can construct a concave symplectic filling $(Z_0,\omega_0)$ of $(Y',\xi')$ from an open book decomposition, by first performing $0$-surgery on the binding $B$ and then extending the symplectic form across a Lefschetz fibration with concave boundary $Y'_0(B)$.  Kronheimer and Mrowka \cite[Lemma~11]{km-p} observed that one can take $H_1(Z_0)=0$ by choosing a collection of vanishing cycles which generate the homology of the base, but the same argument with $\pi_1$ instead of $H_1$ allows us to take $Z_0$ to be simply connected.

With this in mind, we write $Z = H \cup_{Y'} Z_0$, and note that $Z$ contains a surface of positive self-intersection: this is the core of the handle $H$, glued to a Seifert surface for the attaching curve $T_{2,5}$ in $Y$.   Thus we have a closed symplectic manifold $W = X \cup_Y Z$ such that
\begin{itemize}
\item both $X$ and $Z$ are simply connected, and
\item both $b^+_2(X)$ and $b^+_2(Z)$ are strictly positive.
\end{itemize}
We choose a class $w\in H^2(W;\Z)$ and a symplectic form on $W$ with integral cohomology class, and if the latter class has Poincar\'e dual $h \in H_2(W;\Z)$ then we know from \cite{km-excision,sivek-donaldson} that the Donaldson invariants
\[ D^w_W(h^k) \]
are nonzero for all large enough $k$ in some residue class mod 4.

On the other hand, the assumption that $Y$ is $\SU(2)$-abelian means that all critical points of the Chern--Simons functional on $Y$ are reducible, and in fact the Chern--Simons functional on $Y$ is Morse--Bott since $\pi_1(Y)$ is cyclically finite \cite[Corollary~4.5]{bs-stein}.  Under these circumstances, Austin and Braam \cite[Proposition~6.3]{austin-braam} proved that any Donaldson invariant of the form $D^w_W(x^k)$ must vanish, where $x \in H_2(W;\Z)$ is a class whose restriction to $H_1(Y)$ is zero.  Letting $n = |H_1(Y)|$, the restriction of $x=nh$ to $H_1(Y)$ is certainly zero, so the Donaldson invariant
\[ D^w_W\big( (nh)^k \big) = n^k \cdot D^w_W(h^k) \]
must be zero as well, but this contradicts the nonvanishing result mentioned above.  It follows that we must have had $b^+_2(X) = 0$ after all.
\end{proof}

\begin{remark}
In the proof of Proposition~\ref{prop:fillings}, we can choose the class $w \in H^2(W;\Z)$ to be nonzero mod $2$ on each of $X$ and $Z$, and then the desired vanishing for the invariants $D^w_W(h^k)$ follows equally well from the ``dimension counting argument'' outlined by Donaldson following the statement of \cite[Theorem~4.9]{donaldson-polynomial}.  The key facts we need to carry out this argument are that every flat $\SU(2)$ connection $A$ on $Y$ is reducible and nondegenerate (meaning that $H^1_A(Y) = 0$), or in other words that $Y$ is $\SU(2)$-abelian and $\pi_1(Y)$ is cyclically finite.  This allows us to glue the relevant moduli spaces of instantons over $X$ and $Z$, each with the same asymptotic limit at $Y$, as described for example in \cite[\S4.4.1]{donaldson-book}.
\end{remark}

We remark that Proposition~\ref{prop:fillings} should apply equally well to fillings that are not simply connected, but we do not need this more general statement here.  We also observe the following corollary, cf.\ \cite[Theorem~2.1]{kmos} or \cite[Theorem~1.4]{osz-genus}.

\begin{corollary} \label{cor:taut-foliations}
Let $Y$ be an $\SU(2)$-abelian rational homology sphere.  If $\pi_1(Y)$ is cyclically finite, then $Y$ does not admit a co-orientable taut foliation.
\end{corollary}

\begin{proof}
If $Y$ has a co-orientable taut foliation then we can perturb it to a weakly fillable contact structure \cite{eliashberg-thurston,bowden-approximating,kazez-roberts-approximating}, and use \cite{eliashberg-cap} to construct a simply connected weak symplectic filling of $(Y,\xi)$ with $b^+ > 0$, exactly as in the proof of \cite[Proposition~15]{km-p}.  This contradicts Proposition~\ref{prop:fillings}.
\end{proof}

\section{Homology of orders 3 and 5} \label{sec:35}

We can now prove the second case of Theorem~\ref{thm:heegaard-genus-2}, which asserts that if $Y$ has Heegaard genus 2 and first homology of order 3, then it is not $\SU(2)$-abelian.  Here we will need the fact, proved in \cite{bs-trefoil}, that the dimension of $\KHI$ is enough to detect the trefoils.

\begin{lemma} \label{lem:su2-simple-det-3}
If $K \subset S^3$ is an $\SU(2)$-simple knot of determinant $3$, then $K$ is a trefoil.
\end{lemma}

\begin{proof}
Proposition~\ref{prop:odd-prime-determinant} says that $\dim_\Q \KHI(K) = 3$, and by \cite[Theorem~1.6]{bs-trefoil} the only knots with three-dimensional instanton knot homology are the trefoils.
\end{proof}

\begin{proof}[Proof of Theorem~\ref{thm:heegaard-genus-2} in the case $|H_1(Y)|=3$]
Suppose that $Y$ has Heegaard genus 2 and is $\SU(2)$-abelian, and that $H_1(Y)$ has order $3$.  By Proposition~\ref{prop:dcover-map}, we can write $Y \cong \dcover(K)$ for an $\SU(2)$-simple 3-bridge knot $K$ of determinant 3.  But Lemma~\ref{lem:su2-simple-det-3} says that there are no such knots, since the trefoils have bridge index 2, so such $Y$ cannot exist.
\end{proof}

We cannot quite prove the analogue of Lemma~\ref{lem:su2-simple-det-3} for knots of determinant $5$, but we can achieve a partial characterization.  We recall a bit of terminology first: an \emph{instanton L-space knot} is a knot in $S^3$ on which some positive Dehn surgery produces an instanton L-space.  Such knots are known to be fibered and strongly quasipositive \cite[Theorem~1.15]{bs-lspace}.

\begin{lemma} \label{lem:su2-simple-det-5}
Let $K$ be an $\SU(2)$-simple knot of determinant $5$ other than the figure eight or $T_{\pm2,5}$.  Then, after possibly replacing it with its mirror, $K$ is an instanton L-space knot, hence fibered and strongly quasipositive, with Seifert genus $g=g(K)$ at least 3 and signature $\sigma(K) \in \{-4,0,4\}$.  Moreover, its branched double cover $Y = \dcover(K)$ bounds a simply connected Stein domain $X$ with
\[ b^+_2(X) = g + \tfrac{1}{2}\sigma(K) \geq 1. \]
\end{lemma}

\begin{proof}
Proposition~\ref{prop:odd-prime-determinant} says that $|\sigma(K)| \leq 4$ and $\sigma(K) \equiv \det(K)-1 \equiv 0 \pmod{4}$, so $\sigma(K)$ must be either $0$ or $\pm4$; and that $\dim_\Q \KHI(K) = 5$.  We will use the Alexander decomposition of $\KHI(K)$ to understand more about $K$.

First, we suppose that $\dim_\Q \KHI(K,g) > 1$.  Then since $\dim_\Q \KHI(K) = 5$ we must have
\[ \KHI(K;\Q) \cong \Q^2_{g} \oplus \Q^{\vphantom{2}}_0 \oplus \Q^2_{-g}, \]
where the subscripts denote the Alexander grading.  But then Theorem~\ref{thm:simple-alexander} and the fact that $\frac{1}{2}\sigma(K)$ is even would tell us that
\[ \Delta_K(t) = 2(-t)^g + 1 + 2(-t)^{-g}, \]
which is impossible since $\Delta_K(1)=1$.  It follows that $\dim_\Q \KHI(K,g)=1$ and that $K$ is fibered.  Moreover, if $K$ were fibered of genus 1 then it would be the figure eight (the trefoils having determinant 3), but we have excluded this possibility.  Therefore $K$ must have genus $g \geq 2$.

Now since $K$ is fibered we can apply \cite[Theorem~1.7]{bs-trefoil} to show that $\KHI(K,g-1) \neq 0$, and since $\KHI(K)$ has total dimension 5, we have
\[ \KHI(K;\Q) \cong \Q_g \oplus \Q_{g-1} \oplus \Q_0 \oplus \Q_{1-g} \oplus \Q_{-g}. \]
According to work of Li and Liang \cite[Theorem~1.4]{li-liang}, this tells us that $K$ is an instanton L-space knot, up to mirroring; we replace $K$ with its mirror as needed to guarantee that it has a positive instanton L-space surgery.  Then $K$ is also strongly quasipositive \cite[Theorem~1.15]{bs-lspace}, as claimed.   Moreover, Farber, Reinoso, and Wang \cite[Corollary~1.4]{frw-cinquefoil} proved (building on a partial characterization in \cite[\S2]{blsy}) that $T_{2,5}$ is the only genus-2 instanton L-space knot, and we have assumed that $K$ is not $T_{2,5}$, so in fact $g \geq 3$.

Boileau, Boyer, and Gordon \cite[Proposition~6.1]{boileau-boyer-gordon-1} proved that since $K$ is strongly quasipositive, if its branched double cover is a Heegaard Floer L-space, then $|\sigma(K)| = 2g(K)$.  We adapt part of their argument to the instanton setting: they push a genus-$g$ Seifert surface $F$ for $K \subset S^3$ into the 4-ball so that $F$ is the intersection of a smooth complex curve with that ball, and then observe that $X=\dcover(F)$ is a Stein manifold \cite{hkp,rudolph-clinks} with boundary $\dcover(K)$.  Then $X$ is simply connected \cite[\S3.1.1]{boileau-boyer-gordon-1}, and has $b_3(X)=0$ because it is Stein, and this implies that $b_2(X)=2g$ by \cite[Remark~3.7]{boileau-boyer-gordon-1}.  Its signature $\sigma(X)$ is equal to $\sigma(K)$ \cite[Lemma~1.1]{kauffman-taylor}, so we determine that $X$ gives a Stein filling of $\dcover(K)$ satisfying
\begin{align*}
b_2^+(X) &= g+\tfrac{1}{2}\sigma(K), &
b_2^-(X) &= g-\tfrac{1}{2}\sigma(K),
\end{align*}
as claimed.  Moreover, since $g \geq 3$ and $\sigma(K) \geq -4$ we conclude that $b_2^+(X)$ is positive.
\end{proof}

We can now apply Proposition~\ref{prop:dcover-map} to complete the proof of Theorem~\ref{thm:heegaard-genus-2}.

\begin{proof}[Proof of Theorem~\ref{thm:heegaard-genus-2}]
We have already proved the cases where $H_1(Y;\Z)$ has order $1$ (in \S\ref{sec:zhs3}) or $3$, so now we suppose that $Y$ has Heegaard genus 2, that $|H_1(Y;\Z)|=5$, and that $Y$ is $\SU(2)$-abelian.   Then Proposition~\ref{prop:dcover-map} says that $Y$ is the branched double cover of an $\SU(2)$-simple 3-bridge knot $K$, with $\det(K)=5$.  Since $K$ is not a 2-bridge knot, it is neither the figure eight nor the torus knot $T_{\pm2,5}$.

We now apply Lemma~\ref{lem:su2-simple-det-5}, replacing $K$ with its mirror (and hence $Y$ with $-Y$) if necessary to ensure that $K$ is an instanton L-space knot.  The conclusion of this lemma says that $Y$ has a simply connected Stein filling $(X,J)$ satisfying $b^+_2(X) \geq 1$.  But Lemma~\ref{lem:cyclically-finite} guarantees that $\pi_1(Y)$ is cyclically finite, so this contradicts Proposition~\ref{prop:fillings}, and thus such $Y$ cannot exist after all.
\end{proof}

\begin{remark}
Unlike in the cases $\det(K)=1$ and $\det(K)=3$, we have not shown that there are no $\SU(2)$-simple knots of determinant $5$ and bridge index greater than $2$.  What we have shown is that if such a knot $K$ exists, then its branched double cover $Y = \dcover(K)$ cannot be $\SU(2)$-abelian.  We saw in the proof of Proposition~\ref{prop:dcover-map} that this cannot happen if $K$ is a 3-bridge knot, but $K$ could still have bridge index $4$ or greater, as per Remark~\ref{rem:abelian-implies-simple}.
\end{remark}

\section{Meridian-traceless representation varieties and the trefoils} \label{sec:trefoil}

Given a knot $K \subset S^3$, we define the representation variety
\[ \sR(K,i) = \left\{ \rho: \pi_1(S^3 \setminus K) \to \SU(2) \mid \rho(\mu) = i \right\}, \]
where $\mu$ is a fixed meridian.  This carries an action of
\[ U(1) / (\Z/2\Z) = \{e^{i\theta} \} / \{\pm1\} \]
by conjugation, which fixes the unique reducible representation in $\sR(K,i)$ and partitions the irreducible representations into $S^1$ orbits.  In this section, we prove Theorem~\ref{thm:trefoil-recognition-main}, which asserts that there is a single $S^1$ orbit of irreducible representations if and only if $K$ is a trefoil.

\begin{theorem} \label{thm:trefoil-recognition}
If $\sR(K,i) \cong \{ \ast \} \sqcup S^1$, then $K$ is a trefoil.
\end{theorem}

\begin{proof}
Suppose for now that $K$ is $\SU(2)$-simple, so that every irreducible, meridian-traceless representation $\rho: \pi_1(S^3 \setminus K) \to \SU(2)$ has binary dihedral image.  We recall that according to Klassen \cite[Theorem~10]{klassen}, there are exactly $\frac{1}{2}(\det(K)-1)$ such $\rho$ up to conjugacy.  But by hypothesis there is a unique conjugacy class of such $\rho$, so now $K$ is an $\SU(2)$-simple knot with $\det(K) = 3$, and Lemma~\ref{lem:su2-simple-det-3} says that $K$ must therefore be a trefoil.

We now show that $K$ is indeed $\SU(2)$-simple.  Fixing the nontrivial central character
\[ \chi: \pi_1(S^3 \setminus K) \twoheadrightarrow H_1(S^3\setminus K) \cong \Z \to \{\pm 1\}, \]
we can take any representation in $\sR(K,i)$, say
\[ \rho: \pi_1(S^3 \setminus K) \to \SU(2), \]
and consider the representation
\[ \rho'(\gamma) = j \cdot \chi(\gamma) \rho(\gamma) \cdot j^{-1}. \]
We note that $\rho' \in \sR(K,i)$, since $\chi(\mu) = -1$ and $\rho(\mu)=i$ imply that
\[ \rho'(\mu) = j \cdot (-i) \cdot j^{-1} = i. \]
The operation $\rho \mapsto \rho'$ defines an involution not just on $\sR(K,i)$ but on the entire meridian-traceless $\SL_2(\C)$ character variety of $K$, and Nagasato and Yamaguchi \cite[Proposition~3]{nagasato-yamaguchi} proved that the irreducible characters fixed by this involution are precisely the metabelian ones, which are the same as the binary dihedral characters \cite[Proposition~2]{nagasato-yamaguchi}.

If $\rho$ is conjugate to $\rho'$ then they have the same characters, so the above says that $\tr \rho$ is equal to the character of an irreducible binary dihedral representation $\rho_0$, which we can take to have image in $\SU(2)$.  Facts (1) and (2) in the proof of \cite[Proposition~15]{klassen} say that $\rho$ and $\rho_0$ are conjugate in $\SL_2(\C)$, and then that they are conjugate by an element of $\SU(2)$.  Thus $\rho$ can only be conjugate to $\rho'$ if $\rho$ is itself a binary dihedral representation.  But if $K$ is not $\SU(2)$-simple then we can choose some irreducible $\rho \in \sR(K,i)$ which is not binary dihedral, and the corresponding $\rho' \in \sR(K,i)$ will not be conjugate to $\rho$.  In this case $\sR(K,i)$ has at least two distinct $S^1$ orbits, which contradicts the hypothesis that $\sR(K,i) \cong \{\ast\} \sqcup S^1$.  We conclude that $K$ must be $\SU(2)$-simple after all, and hence $K$ is a trefoil as argued above.
\end{proof}

\section{Homology of order 2} \label{sec:order-2}

In this section we will prove Proposition~\ref{prop:su2-simple-det-2}, which classifies $\SU(2)$-simple links of determinant $2$.  We will then use this to prove Theorem~\ref{thm:order-2}, constructing non-abelian $\SO(3)$-representations of arbitrary 3-manifolds $Y$ with Heegaard genus $2$ and first homology of order $2$.

\begin{lemma} \label{lem:det-linking}
If $L \subset S^3$ is a link of $\ell \geq 1$ components, then $\det(L)$ is a multiple of $2^{\ell-1}$.  If $\ell=1$ then $\det(L)$ is odd, while if $\ell=2$ and we write $L = L_1 \cup L_2$, then
\[ \det(L) \equiv 2 \lk(L_1,L_2) \pmod{4}. \]
\end{lemma}

\begin{proof}
Let $Y = \dcover(L)$.  We will assume $\det(L) \neq 0$ for now, since otherwise $2^{\ell-1}$ divides it anyway.  Then $H_1(Y;\Z)$ is finite of order $\det(L)$, and according to the universal coefficient theorem and equation \eqref{eq:homology-bdc}, we have
\[ H_1(Y;\Z) \otimes (\Z/2\Z) \cong H_1(Y;\Z/2\Z) \cong (\Z/2\Z)^{\ell-1}. \]
Writing $H_1(Y;\Z) \cong \bigoplus_i \Z/d_i\Z$ as a sum of cyclic groups, so that $\det(L) = \prod_i d_i$, each $\Z/d_i\Z$ summand contributes a $\Z/2\Z$ summand to the tensor product on the left if $d_i$ is even, and the trivial group otherwise.  Thus there must be exactly $\ell-1$ even values of $d_i$, and it follows that $2^{\ell-1}$ divides $\prod_i d_i = \det(L)$.  If $\ell=1$ then all of the $d_i$ are odd, so their product $\det(L)$ is odd as well.

Now suppose that $\ell=2$ and that $\det(L)$ is arbitrary.  In this case it equals
\[ |H_1(Y;\Z)| = 2\Delta_L(-1,-1), \]
where $\Delta_L$ is the multivariable Alexander polynomial of $L = L_1 \cup L_2$, as shown by Hosokawa and Kinoshita \cite{hosokawa-kinoshita}.  On the other hand, Torres \cite{torres} proved that the linking number $\lk(L_1,L_2)$ satisfies
\[ \Delta_L(1,1) = \pm \lk(L_1,L_2), \]
so in particular
\[ \Delta_L(-1,-1) \equiv \Delta_L(1,1) \equiv \lk(L_1,L_2) \pmod{2}. \]
Thus $\det(L) \equiv 2\lk(L_1,L_2) \pmod{4}$ as claimed.
\end{proof}

In what follows we will let
\[ D = \{e^{i\theta}\} \cup \{e^{i\theta}j\} \subset \SU(2) \]
denote the binary dihedral subgroup of $\SU(2)$.

\begin{lemma} \label{lem:det-2-bd-image}
Let $L \subset S^3$ be a link with $\det(L) \equiv 2 \pmod{4}$, and fix a meridian-traceless representation
\[ \rho: \pi_1(S^3 \setminus L) \to \SU(2) \]
with non-abelian image in the binary dihedral group $D$.  Then $\rho$ sends every meridian of $L$ to an element of the form $e^{i\theta}j$.
\end{lemma}

\begin{proof}
Lemma~\ref{lem:det-linking} tells us that $L$ must have exactly two components, since $\det(L)$ is neither odd nor a multiple of $4$, and that if we write $L = L_1 \cup L_2$ then $\lk(L_1,L_2)$ is odd.  We will consider the Wirtinger presentation associated to some diagram of $L$, which is generated by the meridians around each strand.  From this presentation we know that all of the meridians of $L_1$ are conjugate to each other, and likewise the meridians of $L_2$ are all mutually conjugate.

Both the normal subgroup $N = \{e^{i\theta}\}$ of $D$ and its coset $Nj = \{e^{i\theta}j \}$ are fixed by conjugation in $D$, so either $\rho$ sends all the meridians of $L_1$ to elements of $N$, or it sends them all to elements of $Nj$.  The same is true for $L_2$.  If $\rho$ sends both sets of meridians to $Nj$ then we are done; if instead both sets are sent to $N$ then all of $\Img(\rho)$ lies in $N$, contradicting our assumption that $\Img(\rho)$ is non-abelian.  Thus we will assume without loss of generality that the meridians of $L_1$ are sent to $N = \{e^{i\theta}\}$, while the meridians of $L_2$ are sent to $Nj = \{e^{i\theta}j\}$.

We now walk along $L_1$ in a circle and observe the sequence of values of $\rho$ at each strand.  Each of these values is either $i$ or $-i$, since these are the only purely imaginary elements of $N$.  If two such strands $\mu$ and $\mu'$ of $L_1$ meet at a crossing $c$, where the overcrossing strand has meridian $\mu_c$, then the Wirtinger presentation says that
\[ \rho(\mu') = \rho(\mu_c^{\pm1}) \cdot \rho(\mu) \cdot \rho(\mu_c^{\mp 1}) . \]
If $\mu_c$ belongs to $L_1$ then its image lies in $N$, hence commutes with $\rho(\mu) = \pm i$, and so $\rho(\mu') = \rho(\mu)$.  On the other hand, if $\mu_c$ belongs to $L_2$ then we can write $\rho(\mu_c^{\pm1}) = e^{i\theta}j$ for some $\theta$, and we have
\[ e^{i\theta}j \cdot (\pm i) \cdot (e^{i\theta}j)^{-1} = e^{i\theta}j \cdot (\pm i) \cdot (-je^{-i\theta}) = \mp i \]
so then $\rho(\mu') = -\rho(\mu)$.  In particular, as we travel along $L_1$, the value of $\rho$ at each meridian changes sign every time we cross under $L_2$, and it stays the same otherwise.

In order to assign $\pm i$ to each strand of $L_1$ in a consistent way, it follows that $L_1$ must cross under $L_2$ an even number of times.  Changing one of these undercrossings to an overcrossing changes $\lk(L_1,L_2)$ by $\pm1$; if instead we change all of them, then the total change must be even.  But doing this causes $L_1$ and $L_2$ to be unlinked from each other, since now $L_1$ always crosses over $L_2$.  Therefore $\lk(L_1,L_2)$ differs from $0$ by an even number.  In other words, the linking number $\lk(L_1,L_2)$ is even, but we saw at the beginning of the proof that it must also be odd and therefore we have a contradiction.
\end{proof}

\begin{lemma} \label{lem:dihedral-image-link}
Let $L \subset S^3$ be a link of at least two components, with $\det(L) \neq 0$, and let
\[ \rho: \pi_1(S^3 \setminus L) \to D \subset \SU(2) \]
be a meridian-traceless representation with non-abelian image.  Suppose that $\rho$ sends every meridian of $L$ to an element of the form $e^{i\theta} j$.  Then $\Img(\rho)$ is a finite group of order dividing $2\det(L)$.
\end{lemma}

\begin{proof}
We argue exactly as in the proof of \cite[Theorem~10]{klassen}.  Suppose that we fix a diagram of $L$, and let $\mu_1,\dots,\mu_s$ denote the meridional generators of the corresponding Wirtinger presentation.  If we have a meridian-traceless, non-abelian representation
\[ \rho: \pi_1(S^3 \setminus L) \to D \subset \SU(2) \]
such that $\rho(\mu_r) = e^{i \theta_r} j$ for each $r=1,\dots,s$, then the relations in the Wirtinger presentation are equivalent to a system of linear equations, one for each of $r-1$ crossings of the diagram, which at a given crossing take the form
\[ \theta_a - 2\theta_b + \theta_c \equiv 0 \pmod{2\pi} \]
if the strand corresponding to $\mu_b$ crosses over the strands for $\mu_a$ and $\mu_c$.   We can conjugate $\rho$ by $e^{-i\theta_s/2} \in D$ to replace a solution $(\theta_1,\dots,\theta_{s-1},\theta_s)$ with $(\theta_1-\theta_s,\dots,\theta_{s-1}-\theta_s,0)$, using the relation
\[ e^{-i\theta_s/2} \cdot e^{i\alpha}j \cdot (e^{-i\theta_s/2})^{-1} = e^{i(\alpha-\theta_s)} j. \]
Thus up to conjugacy we are free to take $\theta_s = 0$, and then non-abelian representations will correspond to solutions $(\theta_1,\dots,\theta_{s-1})$ that are not identically zero modulo $\pi$.

The $(s-1)\times(s-1)$ integer matrix $A$ describing this system is obtained from an Alexander matrix by setting $t=-1$, so it has determinant
\[ \det(A) = \pm \Delta_L(-1) = \pm \det(L). \]
In particular, if we write $\Delta = \det(L)$ for readability then any non-zero solution satisfies
\[ (\theta_1,\dots,\theta_{s-1}) \in \frac{1}{\det(A)} (2\pi\Z)^{s-1} = \frac{2\pi}{\Delta} \Z^{s-1}, \]
and in this case the image $\rho(\mu_r)$ of each meridian lies in the set
\begin{equation} \label{eq:finite-binary-dihedral}
\left\{ e^{2\pi i m / \Delta} \,\middle\vert\, 0 \leq m < \Delta \right\} \cup \left\{ e^{2\pi i m/\Delta} \cdot j \,\middle\vert\, 0 \leq m < \Delta \right\}.
\end{equation}
Since $\Delta = \det(L)$ is even, this set is closed under multiplication, hence it is a binary dihedral subgroup of $D$, with order $2\det(L)$.  Then $\Img(\rho)$ is a subgroup of this, so its order divides $2\det(L)$, as promised.
\end{proof}

\begin{remark} \label{rem:dihedral-image-knot}
If $L$ were a knot, then the set \eqref{eq:finite-binary-dihedral} would not be closed under multiplication, because it does not contain $-1$ if $\Delta$ is odd.  But adjoining $-1$ turns it into a subgroup, so in this case we would conclude that $\Img(\rho)$ has order dividing $4\det(L)$ instead.
\end{remark}

\begin{proposition} \label{prop:su2-simple-det-2}
Let $L \subset S^3$ be an $\SU(2)$-simple link of determinant $2$.  Then $L$ is a Hopf link.
\end{proposition}

\begin{proof}
Suppose that $L$ is not a Hopf link.  Xie and Zhang \cite[Theorem~1.5]{xie-zhang-traceless} proved that the following are equivalent:
\begin{itemize}
\item $L$ is not an unknot, a Hopf link, or a connected sum of Hopf links; and
\item there is an irreducible representation
\[ \rho: \pi_1(S^3 \setminus L) \to \SU(2) \]
sending every meridian of $L$ to a traceless element.
\end{itemize}
Now we know by Lemma~\ref{lem:det-linking} that $L$ must be a two-component link, and we have assumed that it is not a Hopf link, so such a representation $\rho$ must exist, with non-abelian image.

Suppose in addition that after possibly replacing $\rho$ with a conjugate, the image $\Img(\rho)$ lies in the binary dihedral group $D$.  Lemma~\ref{lem:det-2-bd-image} says that $\rho$ sends every meridian of $L$ to an element of the form $e^{i\theta} j$, so by Lemma~\ref{lem:dihedral-image-link} we know that $\Img(\rho)$ is a finite group of order dividing $2\det(L) = 4$.  But any such group must be abelian, so we have a contradiction.  Thus $\rho$ cannot be conjugate to a binary dihedral representation, and this proves that $L$ is not $\SU(2)$-simple.
\end{proof}

We are now ready to prove Theorem~\ref{thm:order-2}, asserting that if $Y$ has Heegaard genus $2$ and $H_1(Y)$ has order $2$ then there is a non-abelian representation $\pi_1(Y) \to \SO(3)$.  To do so, we recall how to construct a map
\begin{equation} \label{eq:mt-bdc-so3}
B: \left\{\begin{array}{c}\text{meridian-traceless representations}\\ \rho: \pi_1(S^3\setminus L)\to \SU(2)\end{array} \right\} \to \Hom\big( \pi_1(\dcover(L)), \SO(3) \big).
\end{equation}
We have a short exact sequence
\[ 1 \to \pi_1(X_2) \to \pi_1(S^3 \setminus L) \to \Z/2\Z \to 1, \]
where the map to $\Z/2\Z$ sends the meridian $\mu_i$ of each component $L_i$ (say $i=1,\dots,r$) to $1$, and where $X_2$ is the corresponding double cover of the complement of $L$.  Then we obtain $Y \cong \dcover(L)$ by Dehn filling $X_2$ along each lift of $\mu_i^2$, so that
\[ \pi_1(Y) \cong \frac{\pi_1(X_2)}{\llangle \mu_1^2, \mu_2^2, \dots, \mu_r^2 \rrangle}. \]
Now given a meridian-traceless $\rho$, we know that $\rho(\mu_i^2) = -1$ for each $i$, so
\[ \ad \rho: \pi_1(S^3 \setminus L) \to \SO(3) \]
sends each $\mu_i^2$ to the identity; thus its restriction to $\pi_1(X_2)$ induces a well-defined
\[ \tilde\rho: \pi_1(Y) \to \SO(3). \]
The map \eqref{eq:mt-bdc-so3} is defined by $B(\rho) = \tilde\rho$.

The following is part of the proof of \cite[Proposition~3.1]{zentner-simple}.

\begin{lemma}[\cite{zentner-simple}] \label{lem:rho-vs-B}
Suppose that $\det(L) \neq 0$, and let $\rho: \pi_1(S^3 \setminus L) \to \SU(2)$ be a meridian-traceless representation with non-abelian image.  If $B(\rho)$ has abelian image, then $\rho$ is conjugate to a representation with image in the binary dihedral subgroup
\[ D = \{e^{i\theta}\} \cup \{e^{i\theta}j\} \subset \SU(2). \]
\end{lemma}

The condition $\det(L) \neq 0$ is needed to ensure that $\Img(B(\rho))$ is a \emph{finite} subgroup of $\SO(3)$.

\begin{proof}[Proof of Theorem~\ref{thm:order-2}]
Just as in the proof of Proposition~\ref{prop:dcover-map}, we can appeal to Birman and Hilden \cite[Theorem~1]{birman-hilden-branched} to write
\[ Y \cong \dcover(L) \]
for some 3-bridge link $L$ in $S^3$, with determinant $\det(L) = |H_1(Y;\Z)| = 2$ and hence (by Lemma~\ref{lem:det-linking}) with two components.  If $L$ were the Hopf link then $Y \cong \dcover(L)$ would be $\mathbb{RP}^3$, whose Heegaard genus is only $1$, so now Proposition~\ref{prop:su2-simple-det-2} says that $L$ is not $\SU(2)$-simple.  In particular, there is a meridian-traceless, non-abelian representation
\[ \rho: \pi_1(S^3 \setminus L) \to \SU(2) \]
whose image is not conjugate to a subgroup of the binary dihedral group $D \subset \SU(2)$.

We now use $\rho$ to construct a representation $\tilde\rho = B(\rho): \pi_1(Y) \to \SO(3)$ as in \eqref{eq:mt-bdc-so3}.  Since $\rho$ does not have binary dihedral image, Lemma~\ref{lem:rho-vs-B} says that $\tilde\rho$ cannot have abelian image and so this proves that $Y$ is not $\SO(3)$-abelian.
\end{proof}

\bibliographystyle{myalpha}
\bibliography{References}

\end{document}